\documentclass{amsart}

\pdfoutput=1

\usepackage{lmodern}
\usepackage[margin=1in]{geometry} 
\usepackage{tikz-cd}
\usepackage{graphicx,color}

\usepackage{graphicx}

\newtheorem*{unnumthm}{Theorem~\ref{thm:amovessuffice}}
\newtheorem*{unnumthem}{Theorem~\ref{thm:cone}}

\usepackage{amssymb,amsmath}
\usepackage{mathtools}
\usepackage{mathrsfs}
\usepackage{enumerate}
\usepackage{ragged2e}
\usepackage{enumitem,pifont}
\usepackage{tikz}
\usetikzlibrary{arrows,decorations.pathmorphing,automata,backgrounds}
\usetikzlibrary{backgrounds,positioning}

\usepackage{caption}
\usepackage{subcaption}
\usepackage{float}

\usepackage{egothic}
\usepackage[T1]{fontenc}
\usepackage{yfonts}

\graphicspath{{Figures/}}

\usepackage{calc}

\usepackage[%
    colorlinks=true,
    pdfborder={0 0 0},
    linkcolor=blue,
	hyperfootnotes=false
]{hyperref}

\newcommand{\from}{\colon}  

\newcommand{\numset}[1]{\mathbb{#1}}

\newcommand{\R}{\numset{R}}

\usepackage[american]{babel}
\usepackage[babel, final]{microtype}
\usepackage{amssymb}
\usepackage[all]{xy}

\usepackage{mathtools}

\usepackage{ifpdf}
\usepackage{comment}
\usepackage{multirow}

\usepackage[hang,flushmargin]{footmisc}

\makeatletter
\def\@tocline#1#2#3#4#5#6#7{\relax
  \ifnum #1>\c@tocdepth 
  \else
    \par \addpenalty\@secpenalty\addvspace{#2}%
    \begingroup \hyphenpenalty\@M
    \@ifempty{#4}{%
      \@tempdima\csname r@tocindent\number#1\endcsname\relax
    }{%
      \@tempdima#4\relax
    }%
    \parindent\z@ \leftskip#3\relax \advance\leftskip\@tempdima\relax
    \rightskip\@pnumwidth plus4em \parfillskip-\@pnumwidth
    #5\leavevmode\hskip-\@tempdima
      \ifcase #1
       \or\or \hskip 1em \or \hskip 2em \else \hskip 3em \fi%
      #6\nobreak\relax
    \dotfill\hbox to\@pnumwidth{\@tocpagenum{#7}}\par
    \nobreak
    \endgroup
  \fi}
\makeatother

\usepackage{csquotes}
\usepackage[
	backend=biber,
	style=alphabetic,
	sorting=nyt,
	giveninits=true,
	maxnames=10,
	url=false,
	backref=true
]{biblatex}
\DefineBibliographyStrings{english}{
  backrefpage={$\uparrow$},
  backrefpages={$\uparrow$}
}

\DeclareFieldFormat[article,incollection,inproceedings]{title}{#1}
\DeclareFieldFormat[unpublished]{title}{\textit{#1}}

\usepackage{tikz}
\usetikzlibrary{tikzmark}
\usepackage{tikz-cd}

\makeatletter
\g@addto@macro\@floatboxreset\centering
\makeatother

\allowdisplaybreaks[4]

\addto\extrasamerican{%
}

\newtheorem{theorem}{Theorem}[section]

\newtheorem{corollary}{Corollary}[section]
\newtheorem{proposition}{Proposition}[section]
\newtheorem{lemma}{Lemma}[section]
\newtheorem*{geometric-theorem}{Geometric Input Theorem}

\theoremstyle{definition}

\newtheorem{definition}{Definition}[section]
\newtheorem{example}{Example}[section]
\newtheorem{question}{Question}[section]
\newtheorem{remark}{Remark}[section]

\makeatletter
\let\c@conjecture=\c@theorem
\let\c@corollary=\c@theorem
\let\c@proposition=\c@theorem
\let\c@lemma=\c@theorem
\let\c@definition=\c@theorem
\let\c@example=\c@theorem
\let\c@remark=\c@theorem
\let\c@equation\c@theorem
\let\c@question\c@theorem
\makeatother

\def\makeautorefname#1#2{\expandafter\def\csname#1autorefname\endcsname{#2}}
\makeautorefname{theorem}{Theorem}%
\makeautorefname{conjecture}{Conjecture}%
\makeautorefname{corollary}{Corollary}%
\makeautorefname{proposition}{Proposition}%
\makeautorefname{lemma}{Lemma}%
\makeautorefname{definition}{Definition}%
\makeautorefname{example}{Example}%
\makeautorefname{remark}{Remark}%
\makeautorefname{question}{Question}%
\makeautorefname{claim}{Claim}

\numberwithin{equation}{section}

\newcommand*{\Alphabet}{ABCDEFGHIJKLMNOPQRSTUVWXYZ1234567890}
\newcommand*{\alphabet}{abcdefghijklmnopqrstuvwxyz1234567890}
\newlength\fcaph
\newlength\fdesc
\newlength\factualfontsize
\numberwithin{equation}{section}
\newcounter{commentcounter}

\title{Arc diagrams on 3-manifold spines}

\author[J. Brand]{Jack Brand}
\address{The Australian National University}
\email{\href{mailto:jack.brand@anu.edu.au}{jack.brand@anu.edu.au}}

\author[B. A. Burton]{Benjamin A. Burton}
\address{The University of Queensland}
\email{\href{mailto:bab@maths.uq.edu.au}{bab@maths.uq.edu.au}}
\urladdr{\url{http://www.maths.uq.edu.au/~bab/}}

\author[Z. Dancso]{Zsuzsanna Dancso}
\address{University of Sydney}
\email{\href{mailto:zsuzsanna.dancso@sydney.edu.au}{zsuzsanna.dancso@sydney.edu.au}}
\urladdr{\url{http://zsuzsannadancso.net}}

\author[A. He]{Alexander He}
\address{The University of Queensland}
\email{\href{mailto:a.he@uqconnect.edu.au}{a.he@uqconnect.edu.au}}
\urladdr{\url{https://sites.google.com/view/alex-he}}

\author[A. Jackson]{Adele Jackson}
\address{University of Oxford}
\email{\href{mailto:adele.jackson@maths.ox.ac.uk}{adele.jackson@maths.ox.ac.uk}}
\urladdr{\url{http://adelejackson.com}}

\author[J. Licata]{Joan Licata}
\address{The Australian National University}
\email{\href{mailto:joan.licata@anu.edu.au}{joan.licata@anu.edu.au}}
\urladdr{\url{https://sites.google.com/view/joanlicata/home}}

\keywords{
    link diagram, projection, spine, arc diagram, lightbulb trick, shadow equivalence, shadow link}
\def\subjclassname{\textup{2020} Mathematics Subject Classification}
\expandafter\let\csname subjclassname@1991\endcsname=\subjclassname
\expandafter\let\csname subjclassname@2000\endcsname=\subjclassname
\subjclass{
    57K10, 57M15
    \hfill
    Date: \today
}

\begin{document}

\begin{abstract}
    We develop a theory of link projections to trivalent spines of 3-manifolds.  We prove a Reidemeister Theorem providing a set of combinatorial moves sufficient to relate the projections of isotopic links.  We also show that any link admits a crossingless projection to any special spine and we refine our theorem to provide a set of combinatorial moves sufficient to relate crossingless diagrams.  Finally, we discuss the connection to Turaev's shadow world, interpreting our result as a statement about shadow equivalence of a class of 4-manifolds.
\end{abstract}

\maketitle



\tableofcontents

\section{Introduction}
\label{sec:intro}

Nearly a century after the foundational work of Reidemeister \cite{Reidemeister} and Alexander-Briggs \cite{AlexBriggs}, projections continue to play an essential role in link theory.
Reidemeister-type moves have similarly been used to study objects other than links in the $3$-sphere,
such as braids \cite{Artin}, immersed planar curves~\cite{Arnold, Titus}, links in thickened surfaces up to stabilisation (also known as virtual links) \cite{Kuperberg, Kauffman}, and braided ribbon tubes in $\R^4$ (welded braids) \cite{BrendleHatcher, FennRimanyiRourke}.These examples merely sample the broad body of related work in combinatorial topology.

This paper introduces new techniques for studying links in arbitrary 3-manifolds via projections to 2-dimensional spines.  Spines, which are widely employed in combinatorial topology, are distinguished by their universality: every compact orientable 3-manifold admits a spine, and typically the spine provides sufficient data to reconstruct the original manifold.  Just as a knot in $\mathbb{R}^3$ is determined up to isotopy by its diagram on a plane of projection, we show that diagrams on spines determine knots in arbitrary 3-manifolds, and we prove an analogue of Reidemeister's Theorem relating the projections of isotopic knots.

The choice to complicate the space where the projection lives allows us to greatly simplify the projection.  This philosophy generalises work of Cromwell \cite{Cromwell1998} and Dynnikov \cite{DynnikovThreePage}, who studied crossingless projections of knots to a three-page book.  We define a class of spines -- namely, \emph{trivalent spines} -- that generalises Matveev's special spines~\cite{Matveev}, and we show that for any such trivalent spine and any link $L$, there exists a link isotopic to $L$ whose projection to the spine has no crossings.  Our main theorem establishes a set of combinatorial moves sufficient to relate any two crossingless diagrams of isotopic links  on a  fixed trivalent spine.

\begin{unnumthm} Crossingless projections of isotopic links may be related by a sequence of Finger, Vertex, and Exchange moves.

\end{unnumthm}

We also discuss the connection between our combinatorial moves and Turaev's theory of shadows.  Just as a spine determines a $3$-manifold, a spine decorated with certain labels determines a 4-manifold \cite{Turaev}.  There are well established combinatorial moves on such {\it shadowed polyhedra} that preserve the associated 4-manifold up to PL homeomorphism, and we interpret our theorem about crossingless diagrams as a result about equivalence among a particular class of shadows.

\begin{unnumthem}

 Suppose that $D_1$ and $D_2$ are crossingless projections for isotopic framed links in a $3$-manifold $M$ with trivalent spine $\Sigma$.   Then the shadow cones $CO(\Sigma, D_1)$ and $CO(\Sigma, D_2)$ are shadow equivalent.
\end{unnumthem}

This is particularly interesting in light of the fact that fundamental questions about shadow equivalence remain open.

\subsection*{Data Availability Statement} Data sharing is not applicable to this article as no datasets were generated or analysed during the current study.

\subsection*{Acknowledgements}
This work originates in the MATRIX-MFO Tandem Workshop in Topology. We thank the MATRIX and MFO institutes for hosting the workshop virtually and in person, respectively. We are grateful to Saul Schleimer for his significant early contributions to this research at the Workshop and to Henry Segerman for interesting subsequent discussion. We thank the anonymous referees for their thoughtful comments and suggestions.

AH was supported by an Australian Government Research Training Program Scholarship.
AJ was supported by the Clarendon Fund, the Oxford-Australia Trust and Lincoln College.

\section{Definitions}
Throughout this paper, let $M$ be a compact, connected, oriented 3-manifold.
The following definitions use ideas from~\cite{Matveev}.

\begin{definition}\label{def:subpolyhedron}
 An \emph{admissible polyhedron} is a finite CW-complex $\Sigma$ such that each point $x$ in $\Sigma$ has a neighbourhood homeomorphic to one of the three configurations (see Figure~\ref{fig:spinePointTypes}):
\begin{enumerate}
   \item a disc;
    \item three discs intersecting along a common arc in their boundaries, where the point $x$ is on the common arc; or
    \item the \emph{butterfly}, which is a configuration built from four arcs that meet at the point $x$, with a face running along each of the six possible pairs of arcs.
\end{enumerate}
Given an admissible polyhedron $\Sigma$, we write $\Sigma^{(0)}$ for the points with butterfly neighbourhoods, $\Sigma^{(1)}$ for the points with three-disc neighbourhoods, and $\Sigma^{(2)}$ for the points with disc neighbourhoods.
We refer to the connected components of $\Sigma^{(0)}$, $\Sigma^{(1)}$ and $\Sigma^{(2)}$ as \emph{vertices}, \emph{edges} and \emph{faces}, respectively.
\end{definition}

We are specifically interested in admissible polyhedra that are spines of orientable 3-manifolds, so we introduce some further conditions below.

\begin{figure}[!ht]
    \centering
    \begin{subfigure}[t]{0.3\textwidth}
          \includegraphics[width=\textwidth]{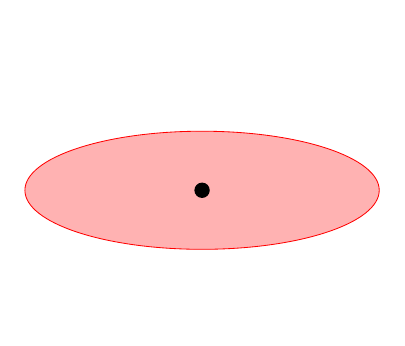}
          \caption{A point in $\Sigma^{(2)}$, with a disc neighbourhood.}
    \end{subfigure}
    \hfill
    \begin{subfigure}[t]{0.3\textwidth}
          \includegraphics[width=\textwidth]{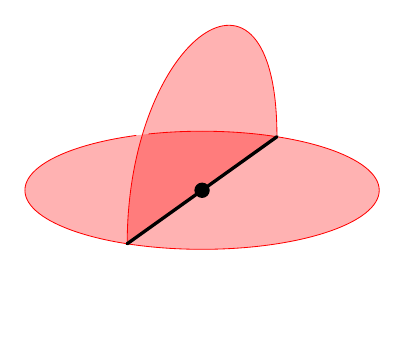}
          \caption{A point in $\Sigma^{(1)}$, whose neighbourhood is three discs touching along an arc.}
    \end{subfigure}
    \hfill
    \begin{subfigure}[t]{0.3\textwidth}
          \includegraphics[width=\textwidth]{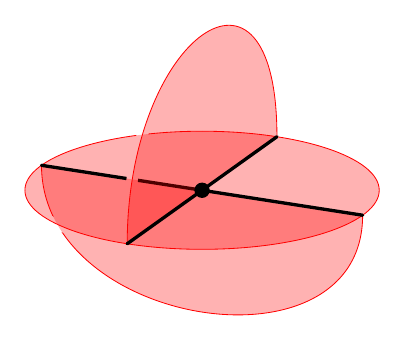}
          \caption{A point in $\Sigma^{(0)}$, with a butterfly neighbourhood.}
    \end{subfigure}
    \caption{The three types of points in an admissible polyhedron $P$.}
    \label{fig:spinePointTypes}
\end{figure}

\begin{definition}\label{def:triv}
A \emph{trivalent spine} $\Sigma$ for a manifold $M$ (possibly with boundary) is an admissible polyhedron tamely embedded in $M$ and satisfying the following conditions:
\begin{enumerate}
    \item\label{product} there is a disjoint union $Z = Z_\Sigma$ of finitely many small open balls in $M$  such that $ M - (\Sigma \cup Z)$ is homeomorphic to $\partial (M - Z)\times [0,1)$;
    \item  $\rho = \rho_\Sigma : (M - Z) \to \Sigma$ is a deformation retraction respecting the product structure in \ref{product}; and
    \item \label{nonempty} $\Sigma^{(1)}$ is not empty.
\end{enumerate}
\end{definition}

\noindent The condition that a trivalent spine $\Sigma$ is an admissible polyhedron implies that $\Sigma^{(2)}$ cannot be empty, which excludes degenerate examples (e.g., allowing a single point to be considered a spine for $S^3$.)
Condition~(\ref{nonempty}) ensures that no component of $\Sigma$ is a closed surface.

From here, we assume all spines are trivalent spines.
\smallskip
When an admissible polyhedron is a spine for a manifold, we can immediately deduce the following.

\begin{lemma}
If $M$ is a compact, connected, oriented 3-manifold, any (trivalent) spine $\Sigma$ for $M$ is itself compact and connected with a finite number of vertices, and all faces of $\Sigma$ have finite type and non-empty boundary.
\end{lemma}

\begin{proof}
The existence of the deformation retraction $\rho_\Sigma$ tells us that $\Sigma$ is compact and connected.
It follows that the components of $\Sigma^{(2)}$ all have non-empty boundary and finite genus. Taking an open cover of $\Sigma$ by neighbourhoods of its vertices shows there are a finite number of points in $\Sigma^{(0)}$.
\end{proof}

A notable class of trivalent spines are \emph{special spines}, defined in~\cite{CaslerSpecialSpines}\footnote{Casler originally used the name ``standard spine''.} and studied in~\cite{Matveev90} and~\cite{Matveev}.
Special spines are trivalent spines with the further restrictions that they have at least one vertex, their edges are all one-cells (so there are no closed loops), and their faces are discs.
They are precisely the spines that arise as duals of triangulations.

\begin{example}
    \emph{Bing's house with two rooms}, denoted $B$ and shown in Figure~\ref{fig:S3trivalentspines} (A), is a special spine (and hence a trivalent spine) for $S^3$. Its complement is a single ball, so $Z = Z_B$ consists of a single small open ball.
An example of a trivalent spine that is \emph{not} a special spine is the {book of three pages}, $T \subset S^3$, which is obtained as the union of one great circle and three distinct spherical caps, as shown in Figure~\ref{fig:S3trivalentspines} (B).
    Here $S^3 - T$ has three components, each of which is a 3-ball, so $Z = Z_T$ consists of three small open balls, one in each component of $S^3 - T$.
    The spine $T$ is significantly simpler and easier to visualise than Bing's house.
\end{example}

\begin{figure}[th!]
\centering
	\begin{subfigure}[b]{0.45\textwidth}

\centering
    \includegraphics[width=0.65\textwidth]{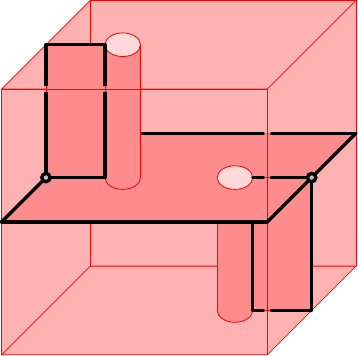}
  \caption{Bing's house with two rooms, a special spine for $S^3$. }
  \end{subfigure}
  \hfill  
  	\begin{subfigure}[b]{0.45\textwidth}
  	\centering
    \includegraphics[width=0.65\textwidth]{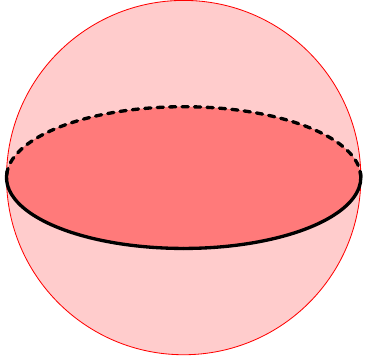}
      \caption{The book of three pages, a trivalent spine for $S^3$.}
  \end{subfigure}
  \caption{Two trivalent spines for $S^3$.}
  \label{fig:S3trivalentspines}
\end{figure}

Fix a trivalent spine $\Sigma$ for $M$. Let $Z = Z_\Sigma$ and fix a deformation retraction $\rho = \rho_\Sigma \from (M - Z) \to \Sigma$.
Given a link $L\subseteq M$,
we define a \emph{diagram} of $L$ with respect to $\Sigma$ by taking the $\rho$-image of $L$ on $\Sigma$:

\begin{definition} A projection of a link $L$ on $\Sigma$ is {\it generic} if $\rho(L)$ is disjoint from the vertices of $\Sigma$ and arcs of $\rho(L)$ intersect each other and the edges of $\Sigma$ only in transverse double points. A {\it diagram} for $L$ is a generic projection with over- and undercrossings recorded at each transverse double point.
\end{definition}

 After fixing $\rho$, we require all diagrams to be generic; as the name suggests, this is easily achieved through a general position argument. We also assume that any isotopies of $L$ are supported in the complement of $Z$.

\section{Moves on projections}

 In order to develop a combinatorial model of link diagrams on spines, we introduce a collection of moves that change the combinatorial type of the projection with respect to the edges of the spine.  In addition to classical Reidemeister moves performed in the interior of faces of $\Sigma$, the moves consist of the Clearing moves, the Finger moves, the Vertex move, and the Exchange move shown in Figure~\ref{fig:pmoves}. Each of the moves corresponds to an isotopy of $L$ in $M$.

\begin{figure}[!ht]
\centering
\includegraphics[width=\textwidth]{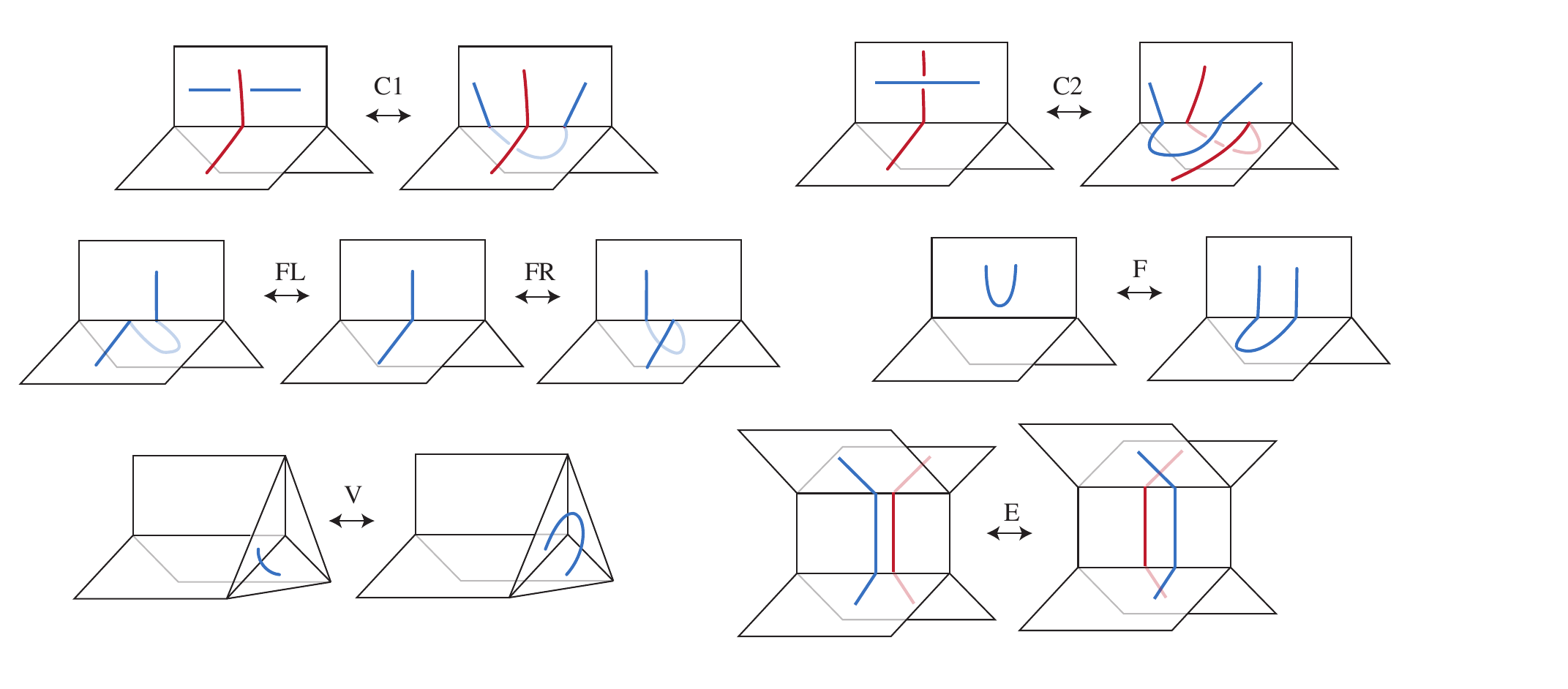}
         \caption{ Top: Clearing moves (C1, C2) where the red/blue crossing is {\em cleared along} the red strand.  Middle: Finger moves (FL, FR, F).  Bottom left: Vertex move (V).  Bottom Right: Exchange move (E).  }\label{fig:pmoves}
\end{figure}

Theorem~\ref{thm:spineReidemeister} below is a Reidemeister Theorem for generic spine diagrams that relates spine projections of isotopic links using these moves and {\em spinal isotopy}:

\begin{definition}
Let $D$ be a generic link diagram for $L$ on a spine $\Sigma$ in $M$.
A \emph{spinal isotopy} of $D$ is an ambient isotopy of $D$ in $\Sigma$.
\end{definition}

Spinal isotopy is analogous to planar isotopy of planar projections of links in $\mathbb{R}^3$. As in that setting, if two diagrams are spinal isotopic, then the corresponding links are ambient isotopic in $M$.

\begin{theorem}\label{thm:spineReidemeister} If $L$ and $L'$ are isotopic in $M$, then their generic spine diagrams differ by a finite sequence of moves -- that is, Clearing, Finger and Vertex moves and Reidemeister moves in the interiors of the faces -- and spinal isotopy.
\end{theorem}

\begin{remark}\label{rmk:Exchange move}

The Exchange move shown in Figure~\ref{fig:pmoves} does not appear in the statement of the theorem. In fact, the Exchange move is a composition of a Reidemeister 2 move, two C1 moves, and two Finger moves, as shown in Figure~\ref{fig:exchangeMoveDerivation}. Note also that while other moves are contained in the neighbourhood of a point, the Exchange move is contained in the neighbourhood of an arc. The benefit of the Exchange move emerges in the next section, when we consider crossingless moves between crossingless diagrams.
\end{remark}

\begin{figure}[th!]
\centering
\includegraphics[width=\textwidth]{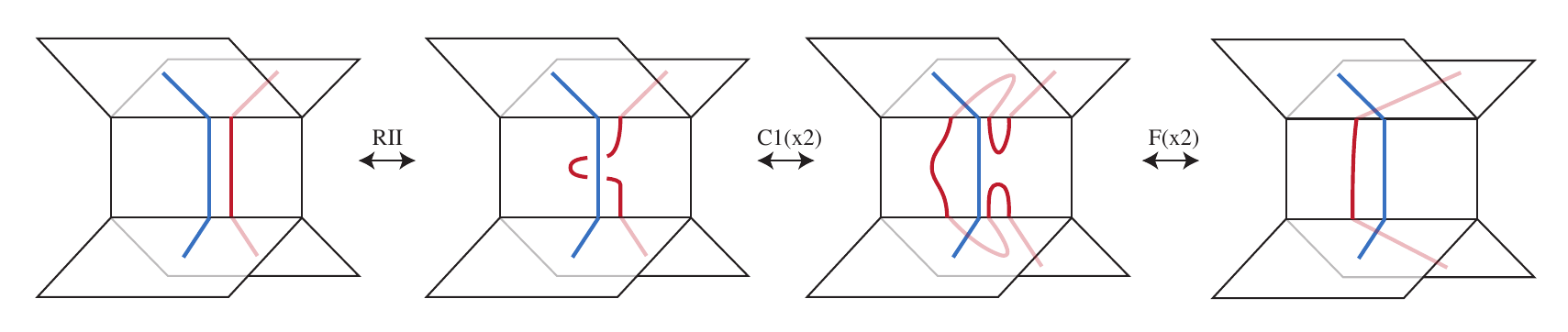}
\caption{The Exchange move follows from the Reidemeister II, Clearing, and Finger moves.}
  \label{fig:exchangeMoveDerivation}
\end{figure}

\begin{proof} Use the product structure on $M\setminus (\Sigma \cup Z)$ coming from Definition~\ref{def:triv} to choose a regular neighbourhood $N(\Sigma)$ of $\Sigma$. Since the boundary of $N(\Sigma)$ is a surface in a 3-manifold, generic isotopies of $L$ in $M$ change the projection of $L$ onto $\partial N(\Sigma)$ by Reidemeister moves.

We examine what happens when $N(\Sigma)$ retracts onto $\Sigma$.  It suffices to consider such a retraction locally in the neighbourhood of a ball centred on $\Sigma$, and we distinguish cases based on whether this centre point is at a vertex, lies on an edge, or lies in a face of $\Sigma$. We further subdivide the isotopy parameter interval so that each time-space ball intersects either at most three strands of $L$ or two strands of $L$ and one edge of $\Sigma$.  We appeal to the classical Reidemeister Theorem to see that this is possible in the first case; the second case is similar, but a segment of an edge replaces a third strand of $L$.

{\em Case 1: A ball centred in a face.} A ball disjoint from $\Sigma^{(1)}$ lives in a product neighbourhood of a disc subset of a face of $\Sigma$, and hence Reidemeister moves on the boundary and isotopies within the interior of the ball both project to Reidemeister moves on the interior of the face.

{\em Case 2: A ball centred on an edge.} In a neighbourhood of a point on $\Sigma^{(1)}$, consider a projection to a disc that identifies two of the faces meeting along the edge.  We treat the edge as a fixed strand and consider isotopies of $L$ inducing Reidemeister II and III moves on this projection.
Figure~\ref{fig:Reidcase2} shows how Reidemeister II moves correspond to Finger moves and Reidemeister III moves correspond to Clearing moves for specific choices of over- and undercrossing data.  The other cases are similar, although sometimes inducing the mirror image of the C2 move (which is shown in Figure~\ref{fig:clearingDirectionChange}). It is straightforward to check that this mirror image is a composite of C2, Finger and Exchange moves.

\begin{figure}[th!]
\centering
\includegraphics[width=\textwidth]{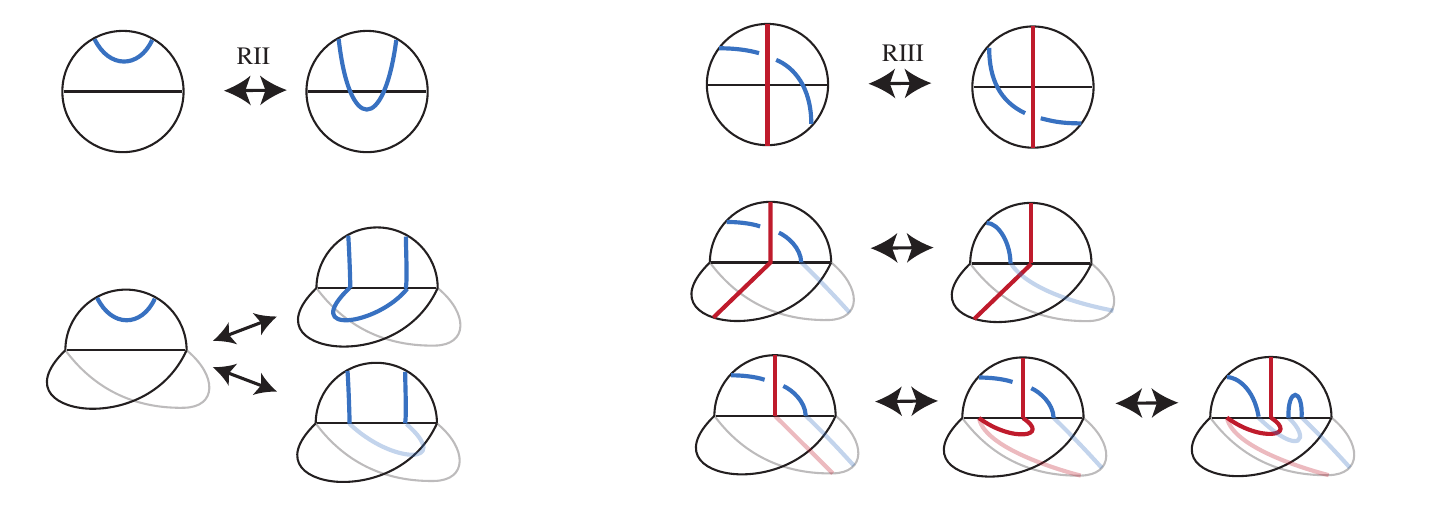}
\caption{In a neighbourhood of a point on $\Sigma^{(1)}$, Reidemeister moves that treat the edge as a strand correspond to compositions of Clearing and Finger moves.}
  \label{fig:Reidcase2}
\end{figure}

\begin{figure}[!ht]
    \centering
    \includegraphics[width=\textwidth]{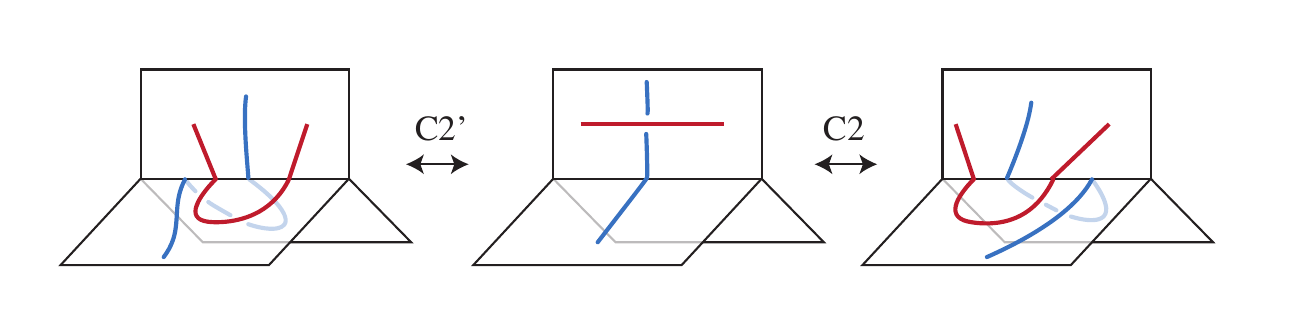}
    \caption{The standard Clearing move from Figure~\ref{fig:pmoves} (C2, right) and the alternative Clearing move (C2', left).}
    \label{fig:clearingDirectionChange}
\end{figure}

{\em Case 3: A ball centred at a vertex.} When the ball is centred at a vertex, after slightly perturbing the isotopy if necessary, we can ensure that no crossing of $L$ ever passes through a vertex, and so can assume that only one strand of the link intersects the ball. Fixing the endpoints of the strand, the strand is either fixed up to isotopy in a face, or moves over a vertex, giving the Vertex move.
\end{proof}

\section{Arc presentations of links}

We call a link diagram on a spine $\Sigma$ an {\it arc diagram} if it has no crossings.  In stark contrast to standard projections to a plane -- where all crossingless diagrams represent unlinked copies of the unknot -- projections to spines can always be made crossingless.

\begin{proposition} Let $M$ be a 3-manifold and $\Sigma$ a trivalent spine for $M$.
Any link in $M$ admits an arc diagram on $\Sigma$.
\end{proposition}

\begin{proof}
Starting with a link projection, we ensure (using Finger moves if necessary) that each component of the projection intersects some edge of the spine.
Each Clearing move eliminates a crossing from the diagram whilst creating no new crossings. Inductively, every crossing can be removed, yielding an arc diagram.
\end{proof}

\begin{remark}
For the rest of this paper we assume that each component of any link diagram intersects $\Sigma^{(1)}$ at least once. As noted above, this can always be achieved via Finger moves.
\end{remark}

Our main result states that any two arc diagrams of isotopic knots are related by a small set of \emph{arc moves} -- namely Finger, Exchange, and Vertex moves -- which involve no crossings.
In other words, not only do arc diagrams exist for every link, but they -- together with the aforementioned arc moves -- also provide a complete combinatorial characterisation of links in a 3-manifold with a fixed trivalent spine.

\begin{definition}\label{def:arceq}
Two link diagrams (not necessarily arc diagrams) are \textit{arc-equivalent} if they differ by a finite sequence of arc moves and spinal isotopies.
\end{definition}

We can now state our main theorem:

\begin{theorem}\label{thm:amovessuffice} On a fixed spine, any two arc diagrams for a link are arc-equivalent.
\end{theorem}

The rest of this section is dedicated to the proof of this theorem.

\begin{question}
    Henry Segerman suggested that the result might hold under a restricted definition of arc-equivalence that excludes the Exchange move.  We would be interested to know if this is true.
\end{question}

\subsection{Clearing forests}\label{sec:forest}
As a first step, we introduce the terminology of {\em clearing forests} to describe the process of turning a link diagram into an arc diagram.
Let $F$ be a face of the spine, and let $T$ denote the restriction of a generic diagram of the link $L$ to $F$.  To each pair $(T, F)$, we associate the embedded graph $\Gamma(T,F)$ in $F$ 
 which has a four-valent vertex at each crossing of $T$ and a univalent vertex at each intersection of $T$ with the boundary of $F$.  Each edge incident to a univalent vertex is said to be {\em boundary-adjacent}.

\begin{definition}\label{def:monotonicClearing}
A {\it monotonic clearing} of a face $F$ is a sequence of Clearing moves that removes all of the crossings of the tangle diagram $T$ on $F$.
A monotonic clearing for $L$ is a sequence of Clearing moves that removes all crossings of $L$.
\end{definition}

Next, we show that monotonic clearings are in bijection with certain collections of forest subgraphs of the face graphs $\Gamma(T,F)$.

\begin{definition}\label{def:ClearingForest}
Let $S$ be a subgraph of $\Gamma(T,F)$, such that $S$ is a forest.
We call $S$, along with a total ordering $<$ on its edges, a {\em clearing forest} for $F$ if it satisfies the following conditions:
\begin{enumerate}
    \item $S$ spans all vertices of $\Gamma(T,F)$.
    \item\label{itm:root} Each connected component of $S$ is a tree where exactly one leaf is a univalent 
    vertex of $\Gamma(T,F)$; we call this distinguished vertex the {\em root}.
    \item\label{itm:ordering} If $e$ and $f$ are two edges in the same connected component of $S$ and the path from $e$ to the root goes through $f$, then $f<e$; we call such a total order {\em admissible}.
\end{enumerate}
It follows from Property~\ref{itm:root} that each connected component of $S$ has a unique orientation towards the root. With this orientation, each four-valent vertex of $T$ is incident to a unique outgoing edge of $S$; we call this the {\em clearing direction} for the vertex.

\begin{figure}[th!]
\centering
\includegraphics[width=\textwidth]{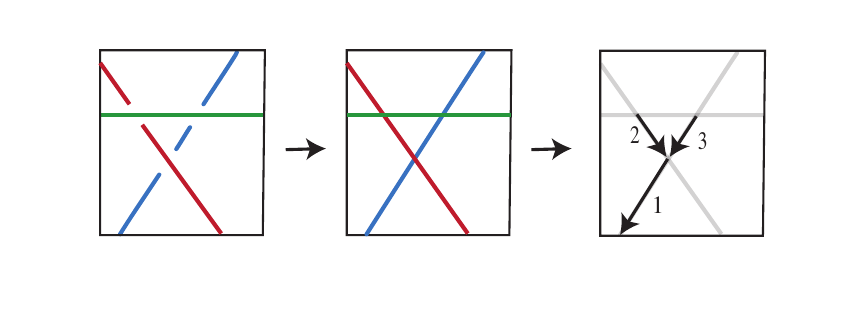}
\caption{A clearing forest associated to a face.  Left: the link diagram $T$ on a face $F$ of $\Sigma$.  Centre: the associated graph $\Gamma(T, F)$. Right: a clearing forest for $F$, with the total order on the bold edges indicated by the integer labels.}
  \label{fig:Forestex}
\end{figure}

Denote the faces of $\Sigma$ by $F_1,\ldots,F_m$.
A {\it clearing forest} for a generic link diagram is a collection $\mathbb{S}=\cup S_i$ where $S_i$ is a clearing forest for $F_i$ and there is a total ordering $<$ on the edges of $\mathbb{S}$ that induces the admissible ordering on each $S_i$. Edges of $\mathbb{S}$ are oriented by their orientations towards the face boundary as edges of $S_i$.
\end{definition}

\begin{remark}\label{rmk:ConnCpt}
A connected component of a clearing forest $\mathbb{S}$ may be either
\begin{enumerate}
\item\label{itm:singletree} a single tree component of a clearing forest $S_i$ on a face; or
\item\label{itm:shareroot} a union of two tree components on adjacent (but not necessarily distinct) faces which share a root.

\end{enumerate}
\end{remark}

\begin{lemma}[Clearing Forest Lemma] \label{lem:ClearFor} Monotonic clearings of a face $F$ are in bijection with clearing forests on $F$. Similarly, monotonic clearings of $L$ are in bijection with clearing forests for $L$.
\end{lemma}

\begin{proof}
We first prove the statement for a fixed face $F$ of the spine.

Given a monotonic clearing, we recursively construct a clearing forest $S$ as follows. Consider the first Clearing move in the monotonic clearing. The first Clearing move is performed along a strand (shown in red in Figure~\ref{fig:pmoves}) that either remains fixed (C1) or changes only by a Finger move supported in an arbitrarily small neighborhood of the point where it intersects $\Sigma^1$ (C2). This strand corresponds to a boundary-adjacent edge of $\Gamma(T,F)$; include this edge in the forest $S$, oriented towards the boundary.

Let $T'$ denote the tangle diagram on $F$ after the first Clearing move is performed. The graph $\Gamma(T',F)$ has one less vertex than $\Gamma(T,F)$, but there is a natural identification of the uncleared edges. Consider the edge of $\Gamma(T',F)$ along which the next Clearing move is performed, and include the corresponding edge of $\Gamma(T,F)$ in $S$, oriented towards the boundary. Continue this process until all crossings on $F$ are eliminated.

Since a monotonic clearing removes every crossing, each vertex of $\Gamma(T,F)$ is ultimately spanned by $S$ and incident to the single outgoing edge along which it is cleared. Following the clearing directions determines a unique, finite-length path from each vertex to a univalent vertex on the boundary of $F$, and all boundary-adjacent edges of $S$ are oriented towards the boundary. It follows that each connected component of $S$ is a tree with a unique boundary-incident leaf, as required. The ordering of the edges is given by the order in which the Clearing moves occur, which satisfies Condition~\ref{itm:ordering} of Definition~\ref{def:ClearingForest}.

On the other hand, each clearing forest $S$ for $F$ determines a monotonic clearing, as follows. Let $e$ be the first edge in the clearing forest, according to the ordering.
By Condition~\ref{itm:ordering}, the edge $e$ is boundary-adjacent. Let $v$ denote the four-valent vertex incident to $e$. Clear the crossing corresponding to $v$ by a Clearing move along $e$, and denote the resulting tangle diagram by $T'$.  Any edge that terminated at $v$ in $S$ naturally corresponds to an edge that terminates at a univalent vertex in $\Gamma(T',F)$; this induces an identification of the edges of $\Gamma(T,F)$ other than $e$ with the edges of $\Gamma(T',F)$. Now perform a Clearing move along the edge of $\Gamma(T',F)$ second in the total ordering of edges in $S$. Inductively, this process removes all the crossings of $T$ in the order given by the edge ordering of $S$.
See Figure~\ref{fig:clearingForestProcedure} for an example of the first step of the procedure.

To prove the statement for the entire link, note that a clearing forest for $L$ is precisely a forest with ordered edges that restricts to a clearing forest $S_i$ for each of the faces $F_i$ of $\Sigma$. Therefore, the same recursive algorithm works to construct a clearing forest for $L$ from a monotonic clearing, and vice versa.
\end{proof}

\begin{figure}[th!]
\centering
\includegraphics[width=\textwidth]{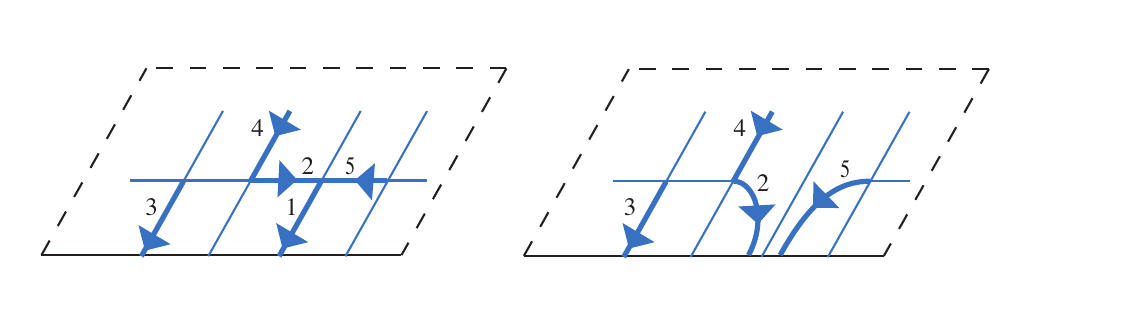}
\caption{Left: The oriented bold edges are part of the clearing forest $S$. Right: The first Clearing move induces a new clearing forest. }
  \label{fig:clearingForestProcedure}
\end{figure}

\subsection{Clearing lemmas}\label{subsec:lemmas}

This section contains a sequence of results addressing the process of turning a generic diagram for the link $L$ into an arc diagram. Many of these statements have a similar form: when clearing a generic diagram via two processes that differ in a specific way, the resulting arc diagrams are arc-equivalent.

The following two lemmas -- the Bulk Finger move and the Finger Change move -- turn out to be very useful tools. They allow us to perform finger moves at the base of trees in clearing forests, depending on which case of Remark~\ref{rmk:ConnCpt} applies.

\begin{lemma}[Bulk Finger]
Let $R$ be a connected component of $\mathbb{S}\cap F_i$
that falls under Case~\ref{itm:singletree} of Remark~\ref{rmk:ConnCpt}.
Let $A_1$ be the diagram resulting from clearing $L$ according to $\mathbb{S}$.
Let $A_2$ be the diagram resulting from clearing $L$ according to $\mathbb{S}$, but after first performing a Finger move at the root of $R$, as illustrated in Figure~\ref{fig:bulkFingerMove}.
Then $A_1$ and $A_2$ are arc-equivalent.
\end{lemma}

\begin{figure}[!ht]
    \centering
    \includegraphics[scale=0.8]{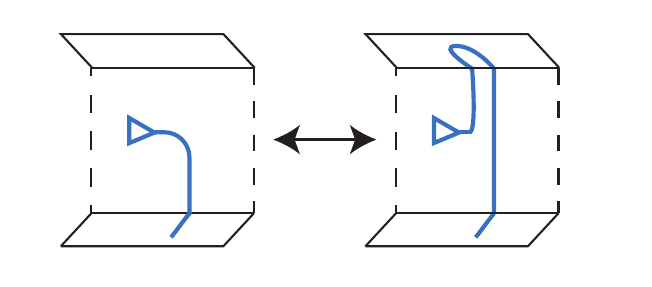}
    \caption{The Bulk Finger move. The tree $R$ is represented by the triangle. The diagram $A_1$ is produced by clearing $R$, as on the left. The diagram $A_2$ is given by clearing the tree after the Finger move shown on the right.
    }
    \label{fig:bulkFingerMove}
\end{figure}

\begin{proof}
First note that $A_1$ and $A_2$ are identical outside of a small neighbourhood of $R$ where the Finger move is performed (i.e., the right hand side of Figure~\ref{fig:bulkFingerMove}); this is a consequence of the fact that $R$ is a connected component of $\mathbb{S}$. Hence, as long as all arc moves that relate $A_1$ and $A_2$ are performed locally in this neighbourhood, we may assume without loss of generality that $R$ is the only component of the clearing forest $\mathbb{S}$.

We proceed by induction on the number of edges in $R$.
If $R$ has zero edges, then the Bulk Finger move is just a Finger move.

Assume the lemma holds whenever $R$ has at most $k$ edges,
and consider the case where $R$ has $k+1$ edges.  Consider the two clearing trees on each side of Figure~\ref{fig:bulkFingerMove}.  In each monotonic clearing, the first Clearing move splits the tree into three components. Note that each of these components has at most $k$ edges, so the inductive hypothesis applies.  Figure~\ref{fig:BFMproof} shows a sequence of Bulk Finger, Exchange, and Finger moves can be applied to relate these two diagrams. This establishes the inductive step, and hence, the lemma.

Note that Figure~\ref{fig:BFMproof} distinguishes two cases, based on whether the first Clearing move is of type C1 or C2.
\end{proof}

\begin{figure}[!ht]
    \centering
    \includegraphics[width=\textwidth]{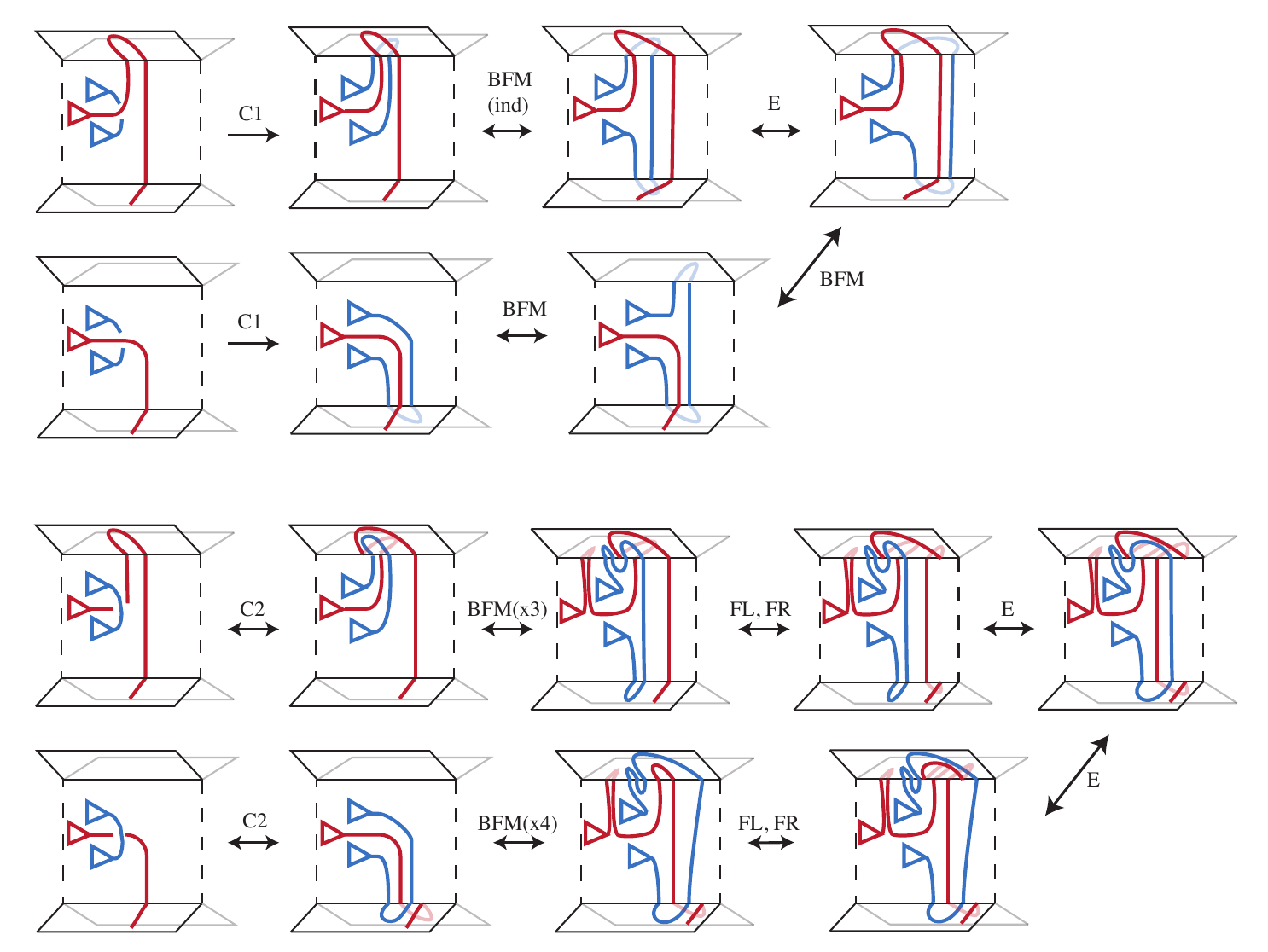}
    \caption{Proof of the Bulk Finger move (BFM).}
    \label{fig:BFMproof}
\end{figure}

The Bulk Finger move is the workhorse of this section.  Heuristically, Finger moves at the root of a clearing tree allow us to isolate individual arc moves from the clearing process.

\begin{lemma}[Finger Change]
\label{lem:fingerchange}
Let $R_1$ and $R_2$ be (possibly empty) clearing trees in a clearing forest $\mathbb{S}$ for $L$ that share a common root, as shown in Figure~\ref{fig:newFingerChange}. (This is Case~\ref{itm:shareroot} of Remark~\ref{rmk:ConnCpt}.) Let $n_1$ and $n_2$ denote the number of edges of $R_1$ and $R_2$, respectively.

Let $A$ be the diagram obtained by clearing $L$ according to the clearing forest $\mathbb{S}$.
Let $B_L$ and $B_R$ denote the diagrams obtained by clearing $L$ according to a clearing forest obtained from $\mathbb{S}$ by performing either a left (FL) or right (FR) Finger move, respectively, at the root of $R_1$ and $R_2$, as shown in Figure~\ref{fig:newFingerChange}.
Then $B_L$ and $B_R$ are each arc-equivalent to $A$.

\begin{figure}[!ht]
    \centering
 \includegraphics[scale=0.7]{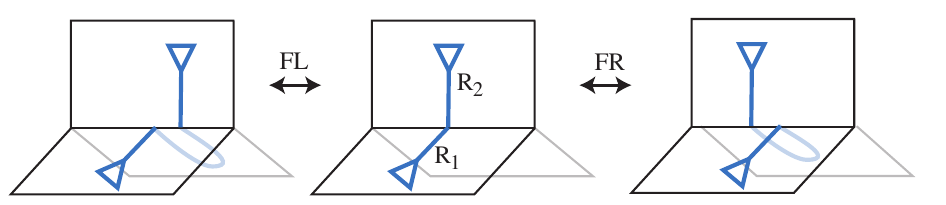}
    \caption{The Finger Change. The triangles represent parts of the trees $R_1$ and $R_2$ not shown. }
    \label{fig:newFingerChange}
\end{figure}
\end{lemma}

\begin{proof}
As with the Bulk Finger move, the diagrams $A$, $B_L$, and $B_R$ are identical outside a small neighbourhood of $R_1 \cup R_2$. Hence, as long as the arc moves relating them are performed locally in this neighbourhood, we may assume that the entire clearing forest consists only of $R_1\cup R_2$.

We proceed by induction on the total number of edges, $n=n_1+n_2$, in $R_1 \cup R_2$.
When $n=0$, the left Finger Change is precisely the FL move, and the right Finger Change is precisely FR.

Assume the lemma holds whenever $n<k$ and consider the case where $n=k$.
Let $\omega$ denote the admissible total order on the edges of $R_1\cup R_2$; we may assume, without loss of generality, that the first edge in this ordering belongs to $R_1$.
Performing the first Clearing move in $\omega$ splits $R_1$ into three smaller clearing trees $Q_1$, $Q_2$ and $Q_3$.

The arc-equivalence between $A$, $B_L$ and $B_R$ is shown in Figure~\ref{fig:fingerChangeTOP}, in the case when the first Clearing move is of type C1, and in Figure~\ref{fig:fingerChangeBOT} in the case when it is of type C2.
The key observation is that each clearing tree $Q_i$ (i=1,2,3) has strictly fewer edges than $R_1$, which means that the inductive hypothesis enables Finger Changes at the roots of these trees.
\end{proof}

\begin{figure}[!ht]
    \centering
    \includegraphics[width=\textwidth]{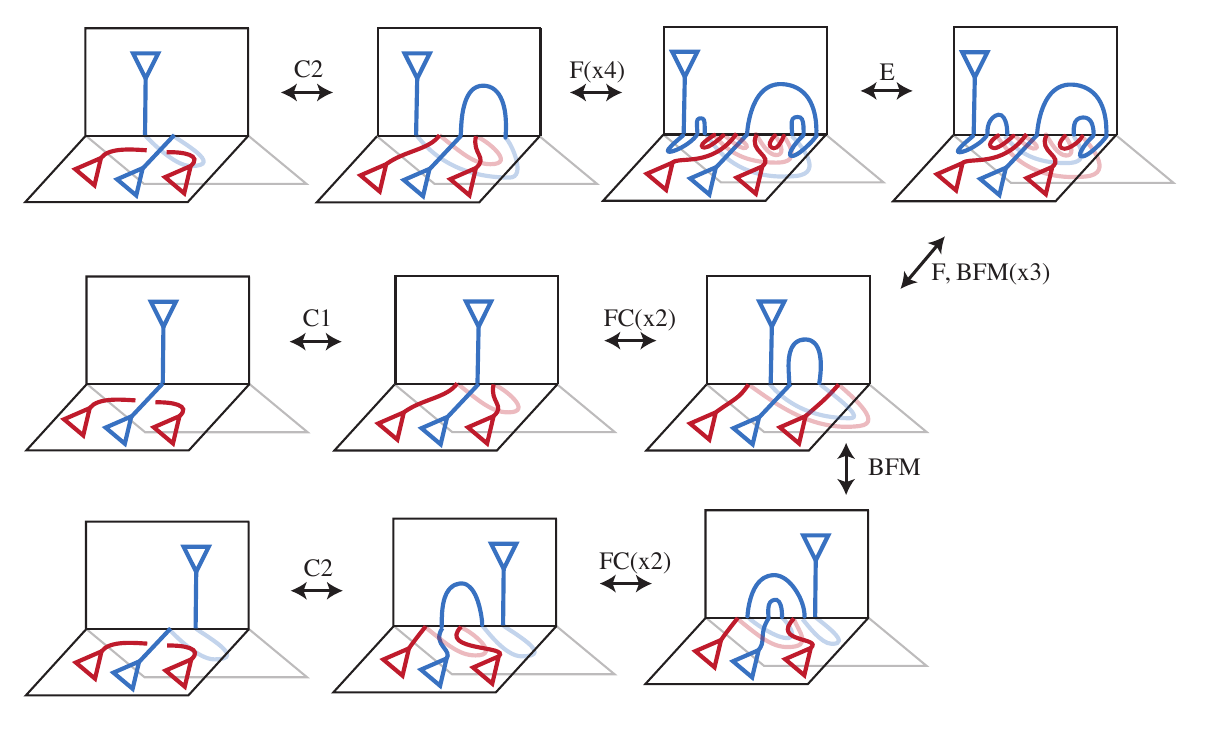}
    \caption{Inductive proof of the Finger Change where the first Clearing move is type C1. The first entry in the middle row shows the clearing forest leading to $A$; in the top row, to $B_R$; and in the bottom row, to $B_L$. Note that in three of the four applications of the inductive Finger Change, one of the two clearing trees is empty.}
    \label{fig:fingerChangeTOP}
\end{figure}

\begin{figure}[!ht]
    \centering
    \includegraphics[width=\textwidth]{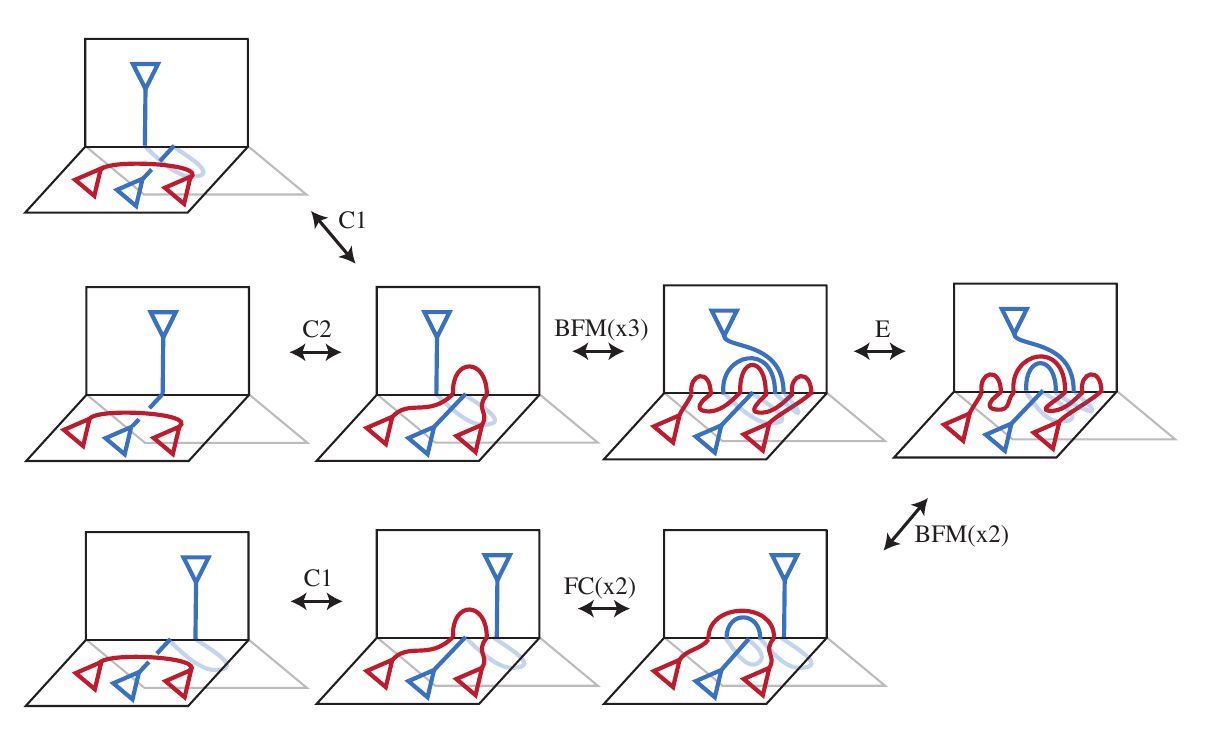}
    \caption{Inductive proof of the Finger Change where the first Clearing move is of type C2. As before, the middle row shows the clearing forest leading to $A$, the top row to $B_R$ and the bottom row to $B_L$.}
    \label{fig:fingerChangeBOT}
\end{figure}

The C2 move shown in Figure~\ref{fig:pmoves} performs a Finger move on the clearing strand, shifting the lower part of the vertical strand to the right.
Alternatively, one could define a Clearing move that shifts this to the left; see Figure~\ref{fig:clearingDirectionChange}.
As a corollary of the Finger Change Lemma~\ref{lem:fingerchange}, we prove the unsurprising result that this choice does not matter:

\begin{corollary}
\label{cor:clearingdirectionchange}
Consider a monotonic clearing of a face $F$ with tangle $T$. Altering the monotonic clearing by replacing any of the C2 moves with a C2' move, as in Figure~\ref{fig:clearingDirectionChange}, results in an arc-equivalent diagram.
\end{corollary}

\begin{proof}
It is enough to prove that altering the direction of a single C2 move yields an arc-equivalent result.

Perform all the Clearing moves ordered before the crossing where the alternative C2 move will be performed. This brings the chosen C2 move to the root of a clearing tree.
 Changing the direction of the Clearing move is the same as performing a Finger move shifting to the left before clearing the crossing.  Thus, by Lemma~\ref{lem:fingerchange}, the diagrams are arc-equivalent. We can then induct on the number of altered Clearing moves to complete the proof.
\end{proof}

 \begin{lemma}[Reordering Lemma]\label{lem:Order}
 If two clearing forests $\mathbb S_1$ and $\mathbb S_2$ for $L$ are identical as graphs and differ only in the total orderings of their edges, then the corresponding cleared diagrams are arc-equivalent.
 \end{lemma}

 \begin{proof}
Let $\omega_i$ be the admissible total order on the edges of the clearing forest $\mathbb{S}_i$ for $i=1,2$. We claim that there exist transpositions $\tau_1,\ldots,\tau_r$ so that $\omega_2=\tau_r\cdots\tau_1\omega_1$, each intermediate permutation $\tau_j\cdots\tau_1\omega_1$ for $1\leq j \leq r$ is admissible, and the indices transposed by $\tau_j$ are adjacent in $\tau_{j-1}\cdots\tau_1\omega_1$, for $1\leq j\leq r$. We prove this claim by induction on $n$, the number of edges in $\mathbb S_1$ (or, equivalently, $\mathbb S_2$). For $n=1$ there is nothing to prove; assume the claim is true when the clearing forests have up to $n-1$ edges.

Number the edges $E(\mathbb S_1)=E(\mathbb S_2)$ of the clearing forests so that $i<_{\omega_1}i+1$ for $1\leq i\leq n-1$.  For convenience, we denote this as $\omega_1=[1,2,\ldots,n]$. Then $\omega_2$ is a permutation $[k_1, k_2,\ldots,k_n]$ of $\{1,2,\ldots,n\}$.

The goal is to pull $k_1$ to the front in $\omega_1$, so that the two permutations have the same first element. Denote the transposition of indices $i$ and $j$ by $(i,j)$, and observe that
$$(k_1,1)\cdots(k_1,k_1-2)(k_1,k_1-1)\omega_1=[k_1,1,2,\ldots,k_1-1,k_1+1,\ldots,n]=:\omega_1'.$$
For each of the transpositions $(k_1,i)$ in the composition above, the edges $k_1$ and $i$ appear in opposite orders in $\mathbb S_1$ and $\mathbb S_2$, both of which are admissibly ordered. Hence, $k_1$ and $i$ are incomparable in the partial order induced by the trees in $\mathbb S_1$ (equivalently, $\mathbb S_2$), and therefore the composition of an admissible permutation with $(k_1,i)$ remains admissible. In other words, all intermediate permutations
$(k_1,j)(k_1,j+1)\cdots(k_1,k_1-1)\omega_1$, for $1 \leq j \leq k_1-1$ are admissible. Note also that $k_1$ and $j$ are adjacent in $(k_1,j+1)\cdots(k_1,k_1-1)\omega_1$.

Now both $\omega_2$ and $\omega_1'$ begin with $k_1$. Clear the edge $k_1$; this results in clearing forests $\overline{\mathbb S_1}$ and $\overline{\mathbb S_2}$ with $n-1$ edges each, which are identical aside from their ordering. Any admissible order on $\overline{\mathbb S_i}$ induces an admissible order on $\mathbb S_i$ by inserting $k_1$ at the front. Therefore, the claim is true by induction.

Thus, it is enough to consider the case where $\omega_1=[1,2,\ldots,n]$ and $\omega_2$ differs from $\omega_1$ via a single transposition $\tau=(i,i+1)$.

{\em Case 1: The edges $i$ and $i+1$ belong to the same tree.} If the edges $i$ and $i+1$ lie in the same tree -- and we know they are incomparable in the partial order -- then, by the time we reach edge $i$ in the clearing process, $i+1$ will be the boundary-adjacent root of a separate component. (Otherwise, $i+1$ would be greater than $i$ in the partial order and they could not be swapped). We conclude that in this case, clearing via $\omega_1$ or $\omega_2$ leads to the same arc diagram.

{\em Case 2: The edges $i$ and $i+1$ belong to different trees.} By possibly performing a Finger Change, we can ensure that the trees containing $i$ and $i+1$ do not share a root; that is, they belong to genuinely distinct connected components of $\mathbb S$. Therefore, swapping the order of $i$ and $i+1$ results in identical arc diagrams.
 \end{proof}

\begin{remark}
A subtlety in the proof of the Reordering Lemma~\ref{lem:Order} is that when two trees on adjacent faces in $\mathbb S$ share a root -- Case~\ref{itm:shareroot} of Remark~\ref{rmk:ConnCpt} -- swapping the order in which they are cleared does change the resulting arc diagram. However, these are arc-equivalent by the Finger Change Lemma~\ref{lem:fingerchange}.
\end{remark}

\begin{lemma}[Bulk Clearing Lemma]\label{lem:bcl}
Let $v$ be a four-valent vertex of $\Gamma(T_i, F_i)$ with at least two boundary-adjacent edges, $e$ and $d$, incident to $v$.

Let $\mathbb{S}(e)$ and $\mathbb{S}(d)$ be two clearing forests that differ only in their choice of outgoing edge from $v$, so that $\mathbb{S}(d)=(\mathbb{S}(e)\setminus e)\cup d$.  Then the arc diagrams associated to clearing $L$ by $\mathbb{S}(e)$ and $\mathbb{S}(d)$ are arc-equivalent.
\end{lemma}

\begin{proof}
Applying Lemma~\ref{lem:Order} (Reordering) allows us to assume that the trees of $\mathbb{S}(e)\cap F_i$ and $\mathbb{S}(d)\cap F_i$ containing $v$ are cleared last. Thus it suffices to consider the case where these trees constitute their entire respective clearing forests.

We also assume that the Clearing moves along $e$ and $d$ are of type C1; if not, then further Finger Changes can be performed to achieve this. Figure~\ref{fig:bulkClearingProof} shows a sequence of Exchange and Bulk Finger moves that establish the arc-equivalence of the arc diagrams associated to clearing by $\mathbb{S}(e)$ and $\mathbb{S}(d)$.  The two cases are distinguished by whether or not $e$ and $d$ are edges lying on the same strand of the diagram.
\end{proof}

\begin{figure}[!ht]
    \centering
    \includegraphics[width=\textwidth]{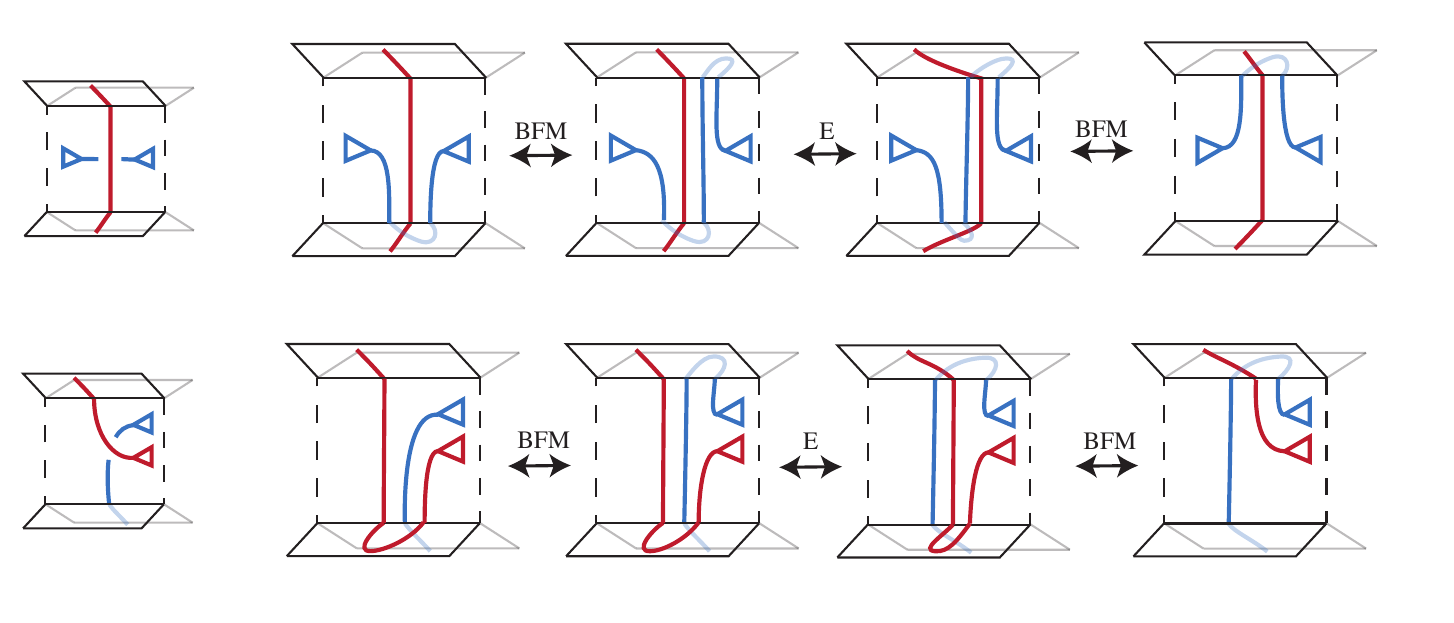}
    \caption{Proof of the Bulk Clearing Lemma. Top: $e$ is the red edge to the bottom and $d$ is the red edge to the top. Bottom: $e$ is the blue edge to the bottom and $d$ is the red edge to the top.}
    \label{fig:bulkClearingProof}
\end{figure}

\subsection{Independence of choice of monotonic clearing}
\label{subsec:pfarcs}

 We now assemble the lemmas of the previous subsection to show that different choices of clearing forest preserve the arc-equivalence class of the resulting arc diagrams.

\begin{proposition}\label{prop:choiceofforest}
Let $L$ be a generic link diagram on the spine $\Sigma$. Let $A_1$ and $A_2$ be two arc diagrams obtained by monotonically clearing the crossings of $L$. Then $A_1$ and $A_2$ are arc-equivalent.
\end{proposition}

\begin{proof}
Let $\mathbb S_1$ and $\mathbb S_2$ denote the clearing forests that yield $A_1$ and $A_2$, respectively. We proceed by induction on the number $n$ of crossings of $L$, or equivalently, the number of vertices of each $\mathbb S_i$. If $n=1$, one can directly check that the results of all clearing procedures are arc-equivalent.

Assume that the proposition holds for all link diagrams with $n-1$ crossings and suppose $L$ has $n$ crossings. The Reordering Lemma~\ref{lem:Order} allows us to choose the order in which crossings are cleared.   If $\mathbb S_1$ and $\mathbb S_2$ have at least one boundary-incident (that is, $\Sigma^{(1)}$-incident) edge in common, choose this edge to be first in both orders and clear the first crossing along it. This produces clearing forests $\mathbb S_1'$ and $\mathbb S_2'$ for the same link diagram with $n-1$ crossings, and the inductive hypothesis applies.

Now suppose $\mathbb S$ is an arbitrary clearing forest for L and $e$ is a boundary-incident edge of $\Gamma(T,F)$ whose interior vertex is labeled $v$.  Label the outgoing edge from $v$ by $d$, and let $\mathbb S(e)=(\mathbb S\setminus d)\cup e$.  We claim that $\mathbb S$ and $\mathbb S(e)$ are arc-equivalent.  Assuming the claim, we address the case of $\mathbb S_1$ and $\mathbb S_2$ having no boundary-incident edges in common by first replacing each of $\mathbb S_i$ by $\mathbb S_i(e)$ (i=1,2) which share a common boundary adjacent edge.  This returns us to the first case, and leaves only the claim to prove.

 Suppose first that $\mathbb S(e)$ has more than one boundary-adjacent edge, and so in particular has some boundary-adjacent edge $f$ that is distinct from $e$. Since $\mathbb{S}$ and $\mathbb{S}(e)$ differ only in one edge, $\mathbb{S}$ must also contain this edge $f$. Applying the Reordering Lemma if necessary, clear the first crossing along $f$ to yield forests with $n-1$ vertices.  The associated arc diagrams are arc-equivalent by the inductive hypothesis.

 Finally, suppose that $d$ and $e$ are the unique boundary-adjacent edges for $\mathbb S$ and $\mathbb S(e)$, respectively.  Then $\mathbb{S}$ and $\mathbb{S}(e)$ exactly satisfy the hypotheses of Lemma~\ref{lem:bcl}, and the associated arc diagrams are arc-equivalent.
\end{proof}

\subsection{Reidemeister moves} Now that we know that all monotonic clearings of a link diagram $L$ are arc-equivalent, we show that Reidemeister moves performed on faces also preserve the arc-equivalence class of the cleared diagram.

\begin{proposition}\label{prop:reidarcs}
If two link diagrams $L_1$ and $L_2$ differ by a single Reidemeister move on the interior of a face $F$ of $\Sigma$, then the arc diagrams formed by clearing all the crossings of $L_1$ and $L_2$ are arc-equivalent.
\end{proposition}

\begin{proof}
There are three cases depending on the type of Reidemeister move. In all cases we assume without loss of generality that $L_1$ and $L_2$ have crossings on the face $F$ only; if there are crossings on other faces, clear those first, by the same clearing forest.

\begin{enumerate}
    \item Reidemeister I: Consider first the diagram with fewer crossings and choose a clearing forest that does not include the edge where R1 creates a loop.  The diagram with the loop has one additional crossing; clear this along a non-loop edge -- which is outgoing from the new vertex -- and order this edge last.  After performing all the lower-ordered Clearing moves, Figure~\ref{fig:RI} provides the required arc-equivalence.

    \begin{figure}[!ht]
    \centering
    \includegraphics[scale=0.6]{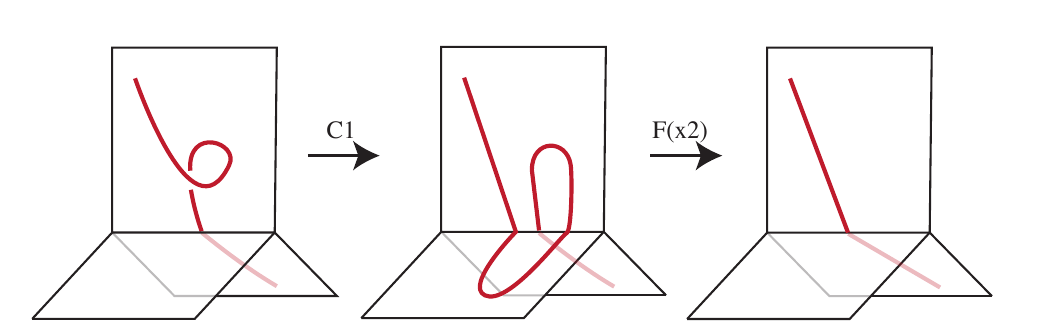}
    \caption{After clearing all crossings away from the R1 move, we can use finger moves to show arc-equivalence.}
    \label{fig:RI}
\end{figure}

    \item Reidemeister II: Let $T$ be the tangle diagram with fewer crossings on $F$, and $T'$ be the diagram with more crossings. Label the new vertices in $\Gamma(F, T')$ by $v$ and $w$. Choose a clearing forest for $\Gamma(F, T)$ that does not involve the arcs where the Reidemeister II move takes place; this is always possible by edge enumeration in a four-valent graph.  Extend this to a clearing forest for $\Gamma(F, T')$ by adding an edge from $v$ to $w$ and an edge from $w$ to a vertex in the existing clearing forest, and as above, choose these to come last in the total order; this is possible because there are no edges oriented towards $v$ or $w$. The sequence of moves in Figure~\ref{fig:RII} completes the argument.

    \begin{figure}[!ht]
    \centering
    \includegraphics[width=\textwidth]{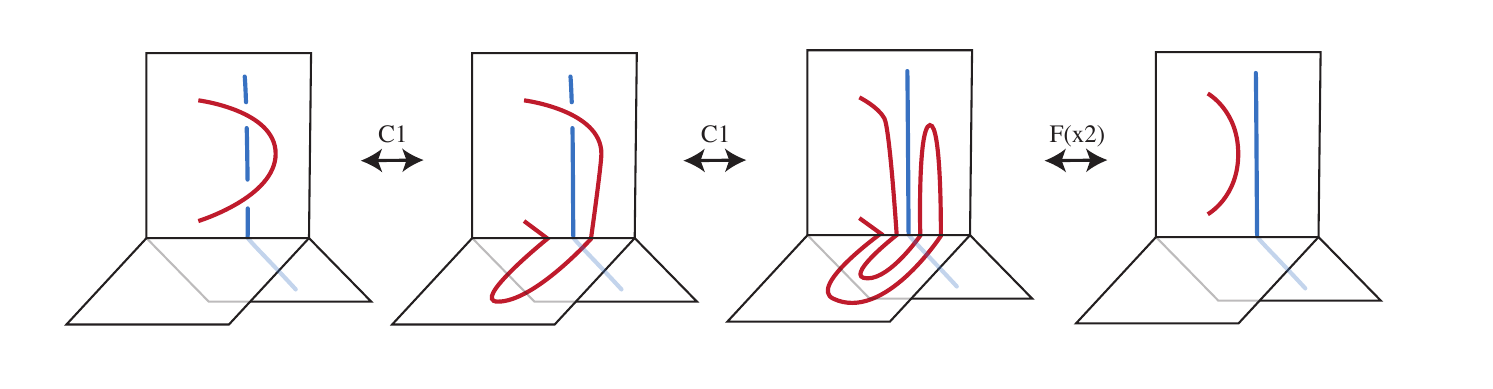}

    \caption{For the R2 move, after clearing the crossings associated to $v$ and $w$ in $T'$, a Finger move recovers the diagram associated to a clearing of $T$.}
    \label{fig:RII}
\end{figure}

    \item Reidemeister III: Again let $T$ and $T'$ denote the two tangles on $F$, differing by a single R3 move. First choose common forests for $\Gamma(F, T)$ and $\Gamma(F, T')$ that do not span the three vertices involved in the Reidemeister III move.  Extend these as shown in Figure~\ref{fig:RIII}, again ordering these edges last.
    The figure shows the required arc-equivalence after the earlier clearings are performed.
    \qedhere
\end{enumerate}

      \begin{figure}[!ht]
    \centering
    \includegraphics[width=\textwidth]{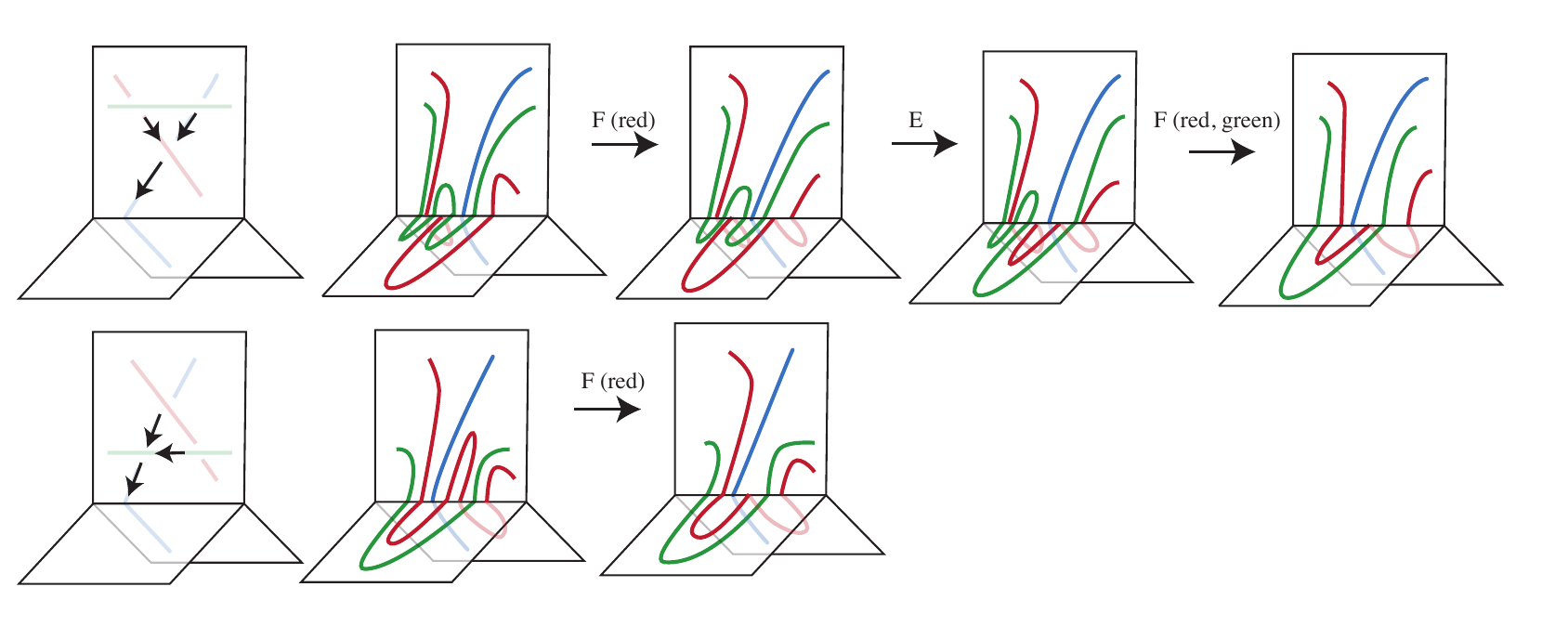}
    \caption{The two rows show partial link diagrams which differ by a single R3 move.  In each case, the figure on the far left indicates the original diagram and a choice of (unordered) clearing forest.  The remaining figures illustrate that after clearing all the crossings as indicated, the diagrams are arc-equivalent.}
    \label{fig:RIII}
\end{figure}

\end{proof}

\subsection{Sufficiency of arc moves}
We are now ready to prove the main theorem: any two arc diagrams of the same link in $(M, \Sigma)$ are arc-equivalent.

\begin{proof}[Proof of Theorem~\ref{thm:amovessuffice}]
Fix a spine $\Sigma$ for the 3-manifold $M$.
Let $A$ and $B$ be two arc diagrams for the same link with respect to $(M, \Sigma)$.
By Theorem~\ref{thm:spineReidemeister}, there is a sequence of link diagrams (not necessarily arc diagrams) from $A$ to $B$ such that each consecutive pair of link diagrams in the sequence is related by one of the following moves: a Reidemeister move in the interior of a face, or a Clearing, Finger, or Vertex move.

If we arbitrarily clear each link diagram in the sequence, it suffices to show that adjacent pairs of cleared diagrams are arc-equivalent: this will provide a sequence of arc moves from $A$ to $B$.
By Proposition~\ref{prop:choiceofforest}, any two choices of monotonic clearing are arc-equivalent, so we can pick the clearing we wish to use at each step.

For each pair of adjacent diagrams, our argument (unsurprisingly) depends on the type of move that relates this pair.
\begin{itemize}
    \item For a Reidemeister move, Proposition~\ref{prop:reidarcs} is the required result.
    \item For a Clearing move, we can choose clearings such that they clear to the \emph{same} arc diagram.
    \item For a Finger move on a strand $e$, we can pick clearings for the two diagrams that are identical until $e$ is boundary-adjacent.
    We can further choose the clearings so that $e$ is the root of a tree, and then by Lemma~\ref{lem:fingerchange} the diagrams are arc-equivalent.
    \item For a Vertex move, we may choose clearing forests for the two link diagrams that do not include the arc involved in the move.
    Then the resulting arc diagrams are related by the original Vertex move.
    \qedhere
\end{itemize}

\end{proof}

\begin{remark}
One natural question is what bound one can give on the number of moves needed in Theorem~\ref{thm:amovessuffice} to get from one arc diagram for a link $L$ on a spine $\Sigma$ to another.
The main challenge is to give a bound on the number of (general) moves needed in Theorem~\ref{thm:spineReidemeister}
to connect two arc diagrams (or, more generally, two link diagrams).
We expect that one only needs to increase the number of general moves by a polynomial factor to convert a general move sequence into a sequence of {\em crossingless} moves.
One might be able to draw on existing work on Reidemeister moves to tackle the main challenge:
Lackenby showed that, given a diagram of the unknot in $S^3$ with $c$ crossings, at most $(236c)^{11}$ Reidemeister moves
are required to simplify this diagram to the trivial one~\cite{LackenbyPolynomialReidemeister}.
Lackenby's proof combined normal surface theory with Dynnikov's work on arc representations~\cite{Dynnikov}.
Since Dynnikov's arc representations are, as mentioned in the Introduction, related to our definition of arc diagrams, similar techniques might be effective to answer the following question:
\end{remark}

\begin{question}
Is there an explicit -- perhaps even a polynomial -- bound on the number of arc moves needed to get between two arc diagrams of a link $L$ to a spine $\Sigma$ in some 3-manifold, in terms of the number of strands of the diagrams and invariants of $L$ and $\Sigma$?
\end{question}

\section{The Lightbulb Trick}

As an application of the results above, we offer a straightforward proof of the classical Lightbulb Theorem: if $K$ is a knot that intersects an $S^2$ fibre of $S^2\times S^1$ once, then $K$ is isotopic to an $S^1$ fibre.

While not difficult, the standard proof requires some visualisation for the process of isotoping the original $K$ to a fibre.  Our proof relies on the results above to realise this argument diagrammatically.

One may naturally associate a spine to a Heegaard splitting by considering the union of the Heegaard surface and a system of cutting discs for each handlebody.  Applying this to the genus-one Heegaard splitting of $S^2\times S^1$, we consider a Heegaard torus and two disjoint meridian discs $D_1$ and $D_2$. The boundary of the discs cuts the torus into two annuli, $A_1$ and $A_2$, and the $S^2$ fibre is isotopic to the union of both discs together with one of these annuli.
Up to isotopy, we can assume that any knot is contained in one of the solid tori.  The hypothesis that $K$ intersects an $S^2$ fibre once lets us further assume that the projection of $K$ to the spine intersects
\begin{itemize}
\item $D_i$ in a single point on $\partial D_i$, for $i=1,2$
    \item $A_2$ in an arc $s$; and
    \item $A_1$ in a tangle.
\end{itemize}

We claim that the number of crossings in the $A_1$ tangle may be reduced while preserving these conditions on the diagram.  It follows that the arc diagram constructed thus is a longitude for both solid tori in the Heegaard splitting, and therefore $K$ is isotopic to an $S^1$ fibre.

The claim is proved in Figure~\ref{fig:lightbulb}.

\begin{figure}[!ht]
    \centering
    \includegraphics[width=\textwidth]{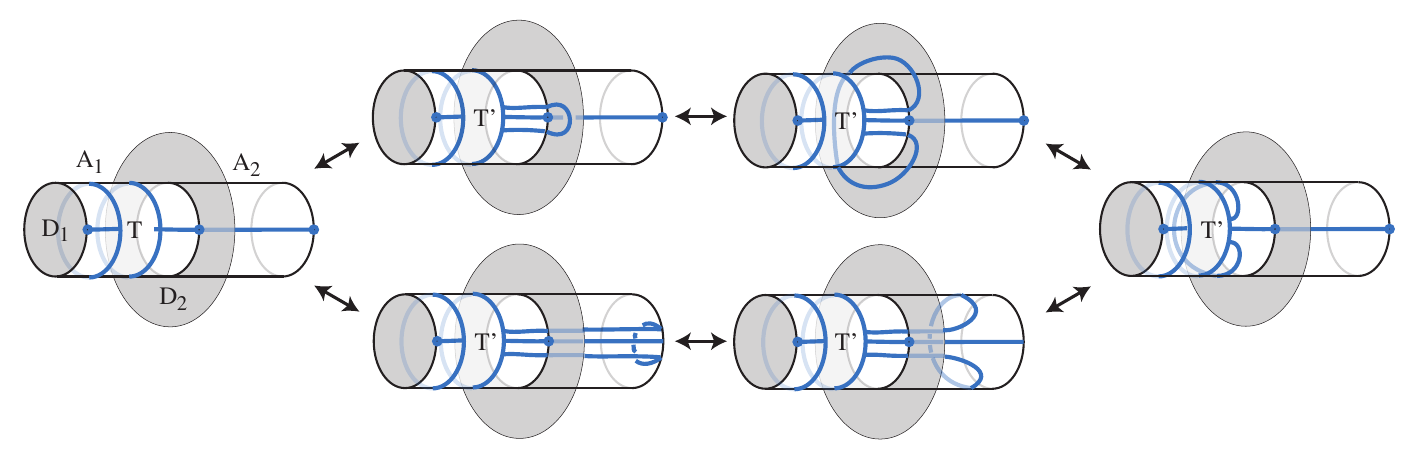}
    \caption{We may reduce the number of crossings in the tangle $T$.  The upper path corresponds to clearing an overcrossing and the lower path, to clearing an undercrossing.}
    \label{fig:lightbulb}
\end{figure}

\section{Shadows of 4-manifolds}

In this final section, we consider a relationship between trivalent spines and 4-manifolds. Turaev's  theory of {\it shadows} considers $2$-polyhedra with additional data: each face is labeled with a {\it gleam} taking values in the integers or half-integers where the topology of the 2-polyhedron determines whether the gleam is integral or half-integral.  The resulting labeled object is called an {\it integer shadowed polyhedron}.  Although these may be studied as purely combinatorial objects, much of the interest in integer shadowed polyhedra is due to the fact that each such polyhedron canonically determines a PL $4$-manifold. Roughly speaking, the construction uses the data of labeled faces to attach 4-dimensional $2$-handles to a 4-dimensional thickening of $\Sigma^{(1)}$; we refer the reader to \cite{CostantinoShadowsIntro} or IX.6.1 of \cite{Turaev} for details.

Turaev introduces a family of combinatorial {\it shadow moves} and shows that two integer shadowed polyhedra related by shadow moves determine PL-homeomorphic $4$-manifolds (Theorem IX.6.2,
\cite{Turaev}).  Turaev also introduces a more general notion of {\it stable shadow equivalence}; two integer shadowed  polyhedra which induce the same $4$-manifold are stably shadow equivalent  (Theorem IX.1.7, \cite{Turaev}), but it is unknown whether they are necessarily shadow equivalent.  In the discussion below, we briefly outline the relevance of our main theorem to this still-open question.

A simple technique for upgrading a spine of an oriented $3$-manifold $M$ to a shadowed polyhedron is to assign a gleam of $0$ to each face.  In this setting, any generic projection $D$  of a framed link $L$ to the spine $\Sigma$ allows us to construct a new shadowed polyhedron known as the {\it shadow cone} $CO(\Sigma, D)$, whereby $D$ becomes the attaching locus for a new face whose gleam is determined by the framing. Turaev shows that the stable shadow equivalence class of the shadow cone is independent of the choice of spine and the representative of the framed isotopy class of $L$ (IX.3.3, \cite{Turaev}).

We consider the case when the projection of $L$ is crossingless.
It is straightforward to show that each of our moves relating arc diagrams can be expressed as a composition of shadow moves on the associated shadow cones, establishing the following result as a  corollary of Theorem~\ref{thm:amovessuffice}:

\begin{theorem}\label{thm:cone} Suppose that $D_1$ and $D_2$ are arc diagrams for isotopic framed links in $M$.    Then the shadow cones $CO(\Sigma, D_1)$ and $CO(\Sigma, D_2)$ are shadow equivalent.
\end{theorem}

\printbibliography

\end{document}